\documentclass[12pt]{amsart}
\usepackage{amssymb, amsmath, amsfonts}
\usepackage[raiselinks=false,colorlinks=true,citecolor=blue,urlcolor=blue,linkcolor=blue,bookmarksopen=true,pdftex]{hyperref}
\newtheorem{theorem}{Theorem}

\newtheorem{lemma}[theorem]{Lemma}
\newtheorem{remark}[theorem]{Remark}

\begin{document}
\baselineskip=15.5pt
\title{Twisted conjugacy in free products}  
\author{Daciberg Gon\c calves}
\address{Instituto de Matem\'amatica e Estat\'istica da Universidade de S\~ao Paulo\\ Departamento de Matem\'atica\\Rua do Mat\~ao, 1010 CEP 05508-090\\S\~ao Paulo-SP, Brasil}
\email{dlgoncal@ime.usp.br}
\author{Parameswaran Sankaran}
\address{Chennai Mathematical Institute , H1 SIPCOT IT Park, Siruseri, Kelambakkam, 603103, India}
\email{sankaran@cmi.ac.in}
\author{Peter Wong}
\address{Department of Mathematics, Bates College, Lewiston, Maine, USA}
\email{pwong@bates.edu}
\thanks{The first author is partially supported by Projeto Tem\'atico-FAPESP Topologia Alg\'ebrica, Geom\'etrica e Diferencial 2016/24707-4 (S\~ao
Paulo-Brazil). The first and third authors thank the IMSc (August 2018) and the CMI, Chennai (December 2019), for their support during their visits. Second and third authors thank the IME-USP, S\~ao Paulo for its support during the authors' visit in February 2019.} 

\date{\today}

\begin{abstract}  Let $\phi:G\to G$ be an automorphism of a group which is a free-product of finitely many groups each of which is freely indecomposable and two of the factors contain proper finite index characteristic subgroups.  We show that $G$ has infinitely many $\phi$-twisted conjugacy classes.  As an application, we show that if $G$ is the fundamental group of a three-manifold that is not irreducible, then $G$ has property $R_\infty$, that is, there are infinitely many $\phi$-twisted conjugacy classes in $G$ for every automorphism $\phi$ of $G$.
\end{abstract}
\maketitle

\section{Introduction}
Let $G$ be an infinite group.  Given an automorphism $\phi: G\to G$, one has an action of $G$ on itself, known as the $\phi$-twisted conjugation, defined as 
$g.x=gx\phi(g^{-1})$. The orbits of this action are 
the $\phi$-twisted conjugacy classes.  Let $\mathcal R(\phi)$ denote the orbit space. 
We denote by $R(\phi)$ the cardinality of $\mathcal R(\phi)$ if it is finite, 
and, when $\mathcal R(\phi)$ is infinite we set $R(\phi):=\infty$ and  $R(\phi)$ is called the Reidemeister number of $\phi$. 
One says that $G$ has the $R_\infty$-property, or that $G$ is an $R_\infty$-group, if $R(\phi)=\infty$ for every 
automorphism $\phi$ of $G$.  The notion of Reidemeister number first arose in the Nielsen-Reidemeister fixed point theory. 
Classifying (finitely generated) groups according to whether or not they have 
the $R_\infty$-property is an interesting problem and has emerged as an active research area that 
has enriched our understanding of finitely generated groups.  

The fundamental group of a closed connected three-dimensional manifold is an important invariant 
of the manifold as it carries a lot of information concerning its topology.  The main motivation for this 
work is to understand which manifolds have the property that their fundamental groups have the 
$R_\infty$-property.  We have not been able to completely answer this question.  However, 
we obtain a very general result showing that a wide class of groups have the $R_\infty$-property.  This  
yields a partial answer, to the above question covering a large class of compact three-manifolds.  

Recall that a closed connected three dimensional manifold $M$ is said to be {\it prime} if $M=M_1\# M_2$ implies that 
at least one of the $M_i$ is a $3$-sphere.  One says 
that $M$ is {\it irreducible} if every embedded $2$-sphere is the boundary of a $3$-disk in $M$.    
\footnote{We work in the PL or smooth category.   Note that every three-manifold admits (unique) PL and smooth structures.}  
Every 
irreducible manifold is prime, but the converse is not true:
$\mathbb {S}^2\times \mathbb {S}^1$ is an example of a prime manifold which is not irreducible.  
If $M$ is irreducible and 
has infinite fundamental group, then the sphere theorem (due to C. D. Papakyriakopoulos) implies that 
$M$ is a $K(\pi,1)$-space.  A fundamental result in three-manifold theory is that every 
closed connected orientable $3$-manifold $M$, can be expressed as a connected sum: 
$M\cong M_1\# \cdots \# M_k$ where each $M_j$ is prime (and not the $3$-sphere).  Moreover 
the decomposition is unique (up to reordering of the factors).  When $M$ is non-orientable, one still has a prime decomposition.  
However, the uniqueness part fails.  If $P$ is the non-trivial $\mathbb S^2$-bundle over $\mathbb S^1$, then  
$P\# N=(\mathbb S^2\times \mathbb S^1)\# N$ when 
$N$ is non-orientable.   In view of this, in the case when $M$ is non-orientable, one may assume that none of 
its prime factors $M_i$ is $\mathbb S^2\times \mathbb S^1$.  With this restriction the uniqueness part is valid. 
See \cite[Chapter 3]{hempel}.  As for any finitely generated group, $\pi_1(M)$ may be decomposed 
as a free product of groups $\pi_1(M)=G_1*\cdots*G_r$ where each $G_i$ is freely indecomposable. 
It turns out that $r=k$ and after a reordering of indices $G_i=\pi_1(M_i), 1\le i\le k$.

Our main result is the following:\\

{\bf Theorem \ref{main-manifold}.} {\em Let $M$ be a non-prime compact connected three-manifold.  Then $\pi_1(M)$ has the $R_\infty$-property.}

The algebraic result which leads to the above as a consequence is the following. 

\begin{theorem}\label{main-group}  Let $k\ge 2$.  
Suppose that $G=G_1*\cdots *G_k$ where (i) each $G_i$ is freely indecomposable, and, (ii)  $G_i$ has a proper 
characteristic subgroups of finite index for $i=1,2$.  Then $G$ has the $R_\infty$-property.  
\end{theorem}

The main tool used in the proof of Theorem \ref{main-group} is Kuro\v s subgroup theorem.   It is well-known that no group is both a nontrivial free product and a nontrivial direct product. See \cite[Observation, p. 177]{ls}.  Thus, if $H=H_0\times H_1, H_0, H_1$ are any two nontrivial groups with 
$H_0$ is a finite group with trivial centre and if $H_1$ is torsionless, then $H$ is freely indecomposable and admits finite index characteristic subgroup, namely $H_1$.   To see this we note that (i) any automorphism of $H$ maps $H_0$ to itself since $H_1$ is torsion-free, the centralizer of $H_1$ in $H$ contains $H_0$, and, (iii) the only 
element of $H_0$ in 
the centralizer of any element $(h_0,h_1)\in H$ is the trivial element. So $H_1$ is characteristic in $H$. 
Therefore we see that the hypotheses on the free factors of $G$ in Theorem \ref{main-group} hold for a large 
family of groups. 

Theorem \ref{main-manifold} 
follows easily from Theorem \ref{main-group} using the fact that  
the fundamental group of a compact three-manifold is residually finite.  (See \cite[Theorem 3.3]{thurston}, \cite{hempel2}). 

We should point out that Fel'shtyn outlined in \cite{fel} the main steps of a proof which shows that finitely generated non-elementary relatively hyperbolic groups have property $R_{\infty}$. This proof relies on group actions on $\mathbb R$-trees and other notions from geometric group theory. Thus Fel'shtyn's result will imply that any finite free product of freely indecomposable finitely generated groups has property $R_{\infty}$ from which Theorem \ref{main-manifold} will follow. On the other hand, Theorem \ref{main-group} does not assume that the free factors are finitely generated and the proof uses elementary techniques from combinatorial group theory. Hence, Theorem \ref{main-group} does not follow from the result of \cite{fel}. For instance, if $G$ is a freely indecomposable 
torsion-free group containing a proper finite index characteristic subgroup and if $H=\bigoplus_{p} \mathbb Z_p$ where $p$ varies over the set of all primes, then $G\ast H$ has property $R_{\infty}$ while $H$ is not finitely generated (see also \cite{sw}).

\section{The $R_\infty$-property of a free product}
Our goal here is to establish the $R_\infty$-property for a free product $G=G_1*\cdots * G_n,n\ge 2,$ for a 
wide class of groups $G_i$.  The main tool will be the Kuro\v s theorem that reveals the structure of a subgroup 
of a free product.  The strategy of proof would be to first establish our goal when all the $G_i$ are finite. Here the case 
$n=2$ is well-known.  We then reduce the general case, under suitable hypotheses on the $G_i$, to the case
of free product of finite groups.

We begin by recalling the Kuro\v s subgroup theorem. Let $G$ be a free product of groups $G=G_1*\cdots *G_n$ 
and let $K$ be a subgroup of $G$.  Then $K$ is itself a free product of groups 
\[K=F_0*H_1*\cdots *H_n\eqno(*)\] 
where each $H_j$ is a free product of a family of subgroups $\{\alpha_{i,j} H_{i,j}\alpha_{i,j}^{-1}\}_{i\in J_j} $ of $G$ 
for suitable elements  $\alpha_{i,j}\in G$ and suitable subgroups  
$H_{i,j}\le G_j, i\in J_j$ for some indexing set $J_j, 1\le j\le n$.  

The following lemma is a standard application of the Kuro\v s subgroup theorem.  We include a proof for the 
sake of completeness.

\begin{lemma} \label{freeproductfinite}
Let $G=G_1*\cdots *G_n$ where 
each $G_i, 1\le i\le n$ is a finite non-trivial group.\\
Then $G$ is virtually free and hence has the $R_\infty$-property if $n\ge 2$.
\end{lemma} 
\begin{proof}
The statement that $G$ is virtually free is trivially valid when $n=1$.  So assume that $n\ge 2$. 
We consider the kernel of projection $\eta: G\to G_1\times \cdots \times G_n$, denoted $K$. 
Note that $\eta$ maps any conjugate of $G_i$ isomorphically onto $G_i$.  Therefore, if $H_i$ is a subgroup of $G_i$ and 
$g\in G$, then $\eta(gH_ig^{-1})$ maps onto $H_i$.  It follows that, writing $K=F_0*K_1*\cdots * K_n$ 
as in $(*)$, we see that $K_i$ are trivial for all $i$. Therefore $K=F_0$ is a free group. Since $G/K=\prod G_i$ is finite, the index of 
$K$ in $G$ is finite.   Since $G$ is finitely generated, the same is true of $K$.

If $n=2$ and $G_1\cong G_2\cong \mathbb Z_2$, then $G$ is infinite dihedral and it is known that 
$G$ has the $R_\infty$-property (see \cite{gw}).  In all other cases, with $n\ge 2$, $K$ is a non-abelian free group of finite rank. 
It follows that $G$ is finitely generated non-elementary word hyperbolic and thus has the $R_\infty$-property by \cite{ll}. 
\end{proof}
 We say that a nontrivial group is {\it freely indecomposable} if it cannot be expressed as a free product of 
 two nontrivial groups. The only nontrivial free group which is freely indecomposable is the infinite cyclic group.

If $\alpha:G\to H$  is an isomorphism and if $C\subset G$ is a characteristic subgroup of $G$, then $\alpha(C)$ is 
characteristic subgroup of $H$ which is independent of the choice of $\alpha$.  Indeed, if $\beta:G\to H$ is another isomorphism then $\beta\circ\alpha^{-1}:H\to H$ 
is an automorphism. Since $\alpha(C)$ is characteristic, we have 
$\alpha(C)=\beta\circ\alpha^{-1}(\alpha (C))=\beta(C)$.  
 
 \begin{lemma} \label{characteristic}
 Let $G=G_1*\cdots *G_n$ where each $G_j, 1\le j\le n,$ is freely indecomposable and not infinite cyclic.  Let 
 $C_j\subset G_j$ be a characteristic subgroup of $G_j$, $1\le j\le n$.  Fix an isomorphism $\alpha_{ij}:G_i\to G_j$ 
whenever $G_i, G_j$ are isomorphic.  
 Then the 
 subgroup $K$ of $G$ generated by the family $\mathcal C$ of subgroups $gC_jg^{-1}, g\alpha_{ij}(C_i)g^{-1} \subset G, g\in G, 1\le i,  j\le n$, is  characteristic in $G$.  
 \end{lemma}
 \begin{proof}  
 Evidently $K$ is normal in $G$ since the family $\mathcal C$ is closed under conjugation.   We need only 
show that the $\mathcal C$ is closed under any automorphism of $G$.  

 Let $\phi:G\to G$.  Consider the subgroup $\phi(G_j).$  Since $G_j$ is freely indecomposable and is not infinite cyclic, 
 the same is true of $\phi(G_j)$.  By the Kuro\v s subgroup theorem, $\phi(G_j)$ is contained in 
$g_j G_{k_j} g_j^{-1}$ for some $k_j\le n$ and $g_j\in G$. Therefore $\phi(G)=\phi(G_1)\ast\cdots\ast\phi(G_n)
\subset g_1G_{k_1}g_1^{-1} \ast\cdots\ast g_nG_{k_n}g_n^{-1}\subset G$. Since $\phi(G)=G$ we must have the equality 
$\phi(G_j)=g_jG_{k_j}g_j^{-1}$ for all $j$.   In particular $\iota_{g_j^{-1}}|_{G_{k_j}} \circ \phi|_{G_j} :G_j\to G_{k_j}$ is an isomorphism, which 
we shall denote by $\kappa_j$. Here $\iota_g$ denotes the inner automorphism $x\mapsto gxg^{-1}$ of $G$.  

Let $A_j\subset G_j$ be any characteristic subgroup of $G_j$.  Then $\kappa_j(A_j)=\alpha_{jk_j}(A_j)$. 
Therefore $\phi(A_j)=g_j(\alpha_{jk_j}(A_j))g_j^{-1}$. 

Taking $A_j$ to be $C_j$ or $\alpha_{ij}(C_i)$, it follows that the family $\mathcal C$ is closed under 
any automorphism of $G$. Hence $K$ is characteristic in $G$. 
 \end{proof}
 
\begin{remark}\label{coherentchoice}
{\em 
(i) In our application, we shall choose the characteristic subgroups $C_j$ so that whenever $G_i$ and $G_j$ are isomorphic, $C_i$  corresponds to $C_j$ under an isomorphism $G_i\to G_j$.  In this case $K\subset G$ is generated as a normal subgroup by the finite collection of subgroups $C_j, 1\le j\le n$. 

(ii) 
We remark that a finite index subgroup of a freely indecomposable group is not necessarily freely indecomposable.
For example $SL(2,\mathbb Z)$ is virtually free with a finite index non-abelian free subgroup but is freely indecomposable.  }
\end{remark}
 
{\it Proof of Theorem \ref{main-group}:}  By relabelling if necessary, we assume that 
(i) $G_1, \ldots, G_n$ are the free factors of $G$ such that either $G_i$ is infinite cyclic or is isomorphic to one of the 
groups $G_1,G_2$, and, (ii) the groups $G_j$ is not isomorphic to any of the groups 
$G_1, G_2, \mathbb Z,$ for $n<j\le k$.  Note that $n\ge 2$.

Let $K$ be the kernel of the natural projection $G\to G_1*\cdots *G_n$ that maps each $G_i$ identically onto $G_i, 1\le i\le n$ and maps $G_j$ to the trivial group for $j>n$.   Then, by Lemma \ref{characteristic}, $K$ is characteristic.    To show that $G$ has the $R_\infty$-property, we need only prove that $G_0:=G_1*\cdots *G_n$ has the $R_\infty$ property.  
 
The proof will be divided into three cases depending on the number of 
groups $G_i, 1\le i\le n,$ that are isomorphic to $\mathbb Z$ being zero, or one or at least two. We shall denote this 
number by $r$.  Relabelling if necessary, we assume that $G_i\cong \mathbb Z$ if $1\le j\le r$ in case $r>0$. 

 Let $C_i\subset G_i$ be a proper finite index characteristic 
subgroup of $G_i$. We assume, as we may, that whenever  $G_i\cong G_j$, then  $C_j$ corresponds to $C_i$ (under 
an isomorphism $G_i\to G_j$).  

{\it Case 1:} Suppose that none of the $G_i$ is infinite cyclic.  Set $\bar{G}_i=G_i/C_i, 1\le i\le n$, and let 
$\bar{G}=\bar{G}_1\ast \cdots\ast\bar{G}_n$. 
Let $K_0$ be the kernel of the natural projection $G_0\to \bar G_0$.  Then $K_0$ is normally generated by the finite collection of subgroups $\{C_j\mid 1\le j\le n\}$.  Hence 
by Lemma 
\ref{characteristic}, $K_0$ is characteristic in $G_0$.  By Lemma \ref{freeproductfinite}, $\bar G_0$ has the $R_\infty$-property.
It follows that $G_0$ has the $R_\infty$-property. 

{\it Case 2:} Suppose that $r>1$ so that $G_i\cong \mathbb Z$ for $1\le j\le r$ and $G_j\not\cong \mathbb Z$ for $j>n$.   If $r=n$, in view of the fact that $n\ge 2$, $G_0$ is a non-abelian free group of finite rank 
and so has the $R_\infty $ property.  

So suppose that $2\le r<n$.  Set $A:=G_1*\cdots *G_r, B:=G_{r+1}*\cdots *G_n$ so that $G_0=A*B$  where $A$ is a {\it non-abelian} 
free group of rank $r$.  Let $K_1$ be the kernel of the projection $G\to A$.   Then $K_1$ is the 
free product of the family of groups $\mathcal C=\{g G_j g^{-1}\mid g\in G, r<j\le n\}$.  
Since each $G_j, j>r$, is indecomposable and not infinite 
cyclic, under any automorphism of $G_0$, $G_j$ is mapped to a conjugate of a $G_i$ isomorphic to $G_j$ where $r<i\le n$.  It follows that 
$\mathcal C$ is stable by any automorphism of $G_0$.  Therefore $K_1$ is characteristic in $G_0$.  Since $A$ is a free non-abelian group of finite rank, it has the $R_\infty$-property.  It follows that $G_0$ also has the $R_\infty$-property.

{\it Case 3:} Suppose that $r=1$, that is, $G_1\cong \mathbb Z$, $G_j\not \cong \mathbb Z$ for $2\le j\le n$.  We consider the canonical projection $\eta: G_0\to G_1*\bar B$ where $\bar B$ is the free product of $G_i/C_i, 2\le i\le n$. 
The kernel $K_3:=\ker(\eta)$ is generated by the collection $\{gC_jg^{-1}\mid g\in G_0, j\ge 2\}$. 
By Lemma 3, $K_3$ is characteristic in $G_0$ in view of our hypothesis on the $C_j$. (See Remark \ref{coherentchoice} (i).)  Now $\bar B$ is {\it nontrivial} and virtually free by Lemma \ref{freeproductfinite}, possibly finite. It follows that $G_1*\bar{B}\cong \mathbb Z*\bar {B}$ has a {\it non-abelian} free group of finite rank as a finite index subgroup.   So $G_1*\bar{B}$ has the $R_\infty$-property.  Since $K_3$ is characteristic in $G_0$, 
it follows that $G_0$ also has the $R_\infty$-property.  

Thus in all cases, $G_0$ has the $R_\infty$-property as was to be shown. \hfill $\Box$

As an immediate corollary, we obtain

\begin{theorem}\label{main-manifold}
Let $M$ be a non-prime compact connected three-manifold.  Then $\pi_1(M)$ has the $R_\infty$-property.
 \end{theorem}
\begin{proof} Since $M$ is not prime, it admits a prime decomposition:
$M=M_1\#\cdots \#M_k$ where each $M_i$ is a prime manifold and $k\ge 2$.  Thus $\pi_1(M)=\pi_1(M_1)*\cdots*\pi_1(M_k)$. 
(See \cite[Chapter 3]{hempel}.)   Note that since $M_i$ is prime, $\pi_1(M_i)$ 
is freely indecomposable in view of \cite[Theorem 7.1]{hempel}.  
Also, it is known that $\pi_1(M_i)$ is residually 
finite as a consequence of the geometrization theorem and the work of Thurston \cite[Theorem 3.3]{thurston}. 
(See also \cite{hempel2}.) So we may take $G_i$ to be $\pi_1(M_i)$ in Theorem \ref{main-group} and we see that 
$\pi_1(M)$ has the $R_\infty$-property. \end{proof}

\begin{remark}{\em 
(i)
The same arguments as above also yield the following: if $M $ is a connected sum  $M_1\#\cdots \# M_r, 
r\ge 2,$ where each $M_j$ is a connected $n$-manifold ($n\ge 3$) and  $M_1$ and $M_2$ each admits
 a non-trivial finite {\it characteristic cover}, then $\pi_1(M)$ has the $R_{\infty}$ property.  (A cover is characteristic if it corresponds to a characteristic 
subgroup of the fundamental group.)

(ii)  Let $P\subset \mathbb N$ be a nonempty proper subset of primes. Suppose that $p\notin P$.   Let $\mathbb Z(P)\subset \mathbb Q$ be subring $\mathbb Z[1/q\mid q\in P]$.  Since $p\notin P$, we have a natural surjective ring homomorphism $\mathbb Z(P)\to \mathbb Z/p\mathbb Z$.  The kernel is a proper characteristic subgroup of 
$\mathbb Z(P)$. (It consists precisely of elements which are expressible as $px, x\in \mathbb Z(P)$.)
Evidently $\mathbb Z(P)$ is freely indecomposable.  It is readily seen that $\mathbb Z(P)$ is not isomorphic to $\mathbb Z(P')$ if $P\ne P'$ so the collection of such groups has cardinality 
the continuum.  
It is known that $\mathbb Z(P)$ is the fundamental group of an open three-manifold $M(P)$, which is in fact an aspherical space. This can be derived from constructing a non-compact $3$-manifold (in fact, the complement of a solenoid in $\mathbb S^3$) whose fundamental group is $\mathbb Q$ (see e.g. \cite[p.209]{em}).  If $M_j:=M(P_j), 1\le j\le k,$ are such manifolds (where we do {\em not} assume that the $P_j$ are pairwise distinct) then their connected sum $M_1\#\cdots \# M_k$ is an open three-manifold 
whose fundamental group has the $R_\infty$-property, provided $k\ge 2$ and at least two sets, say, $P_1,P_2$ are nonempty 
proper subsets of the set of all primes. 
}
\end{remark}

\begin{remark}{\em
(i) Let $M$ be a closed connected three-manifold such that $\pi_1(M)$ is infinite cyclic.   
Then $M$ is prime.  We claim that $M$ is a $2$-sphere bundle over the circle. 
If $M$ is irreducible, then by the sphere theorem, 
$\pi_2(M)$ is trivial.  Since $\pi_1(M)$ is infinite, it is a $K(\mathbb Z, 1)$-space. Therefore it is homotopic to a circle. 
This is a contradiction since $H_3(M;\mathbb Z_2)\cong \mathbb Z_2$.  So $M$ is not irreducible. By \cite[Lemma 3.13]{hempel}
$M$ is a $2$-sphere bundle over the circle.

(ii) When a three-manifold $M$ admits a geometric structure, in some cases it is known whether or not 
$\pi_1(M)$ has the $R_\infty$-property.  For example, 
when $M$ admits spherical geometry, the fundamental group is finite and it is trivial that $\pi_1(M)$ 
does not have the $R_\infty$-property.  On the other hand, when the manifold admits hyperbolic geometry, then 
$\pi_1(M)$ has the $R_\infty$-property as an immediate consequence of the work of Levitt and Lustig \cite{ll}.  In the case 
of $\mathbb S^2\times \mathbb R$-geometry, we see that $\pi_1(\mathbb S^2\times \mathbb S^1)\cong \mathbb Z$ 
does not have $R_\infty$-property whereas $\pi_1(\mathbb RP^3\#\mathbb RP^3)\cong \mathbb Z_2*\mathbb Z_2$ has 
the $R_\infty$-property; see \cite{gw}.  It can be shown that the fundamental groups of 
Seifert fibre spaces have the $R_\infty$-property provided 
the base surface has genus at least $2$.  However the complete classification for geometric three-manifolds will take us too far a field.  It will be carried out in a forthcoming article \cite{gsw}.

In general, by Thurston's geometrization conjecture (Perelman's theorem),  
for any orientable prime three-manifold $M$ that is not geometric, the fundamental group of $M$ is an iterated 
amalgamated free product of fundamental groups of geometric manifolds where the amalgamating group 
is $\mathbb Z^2$.  Our approach to Theorem \ref{main-group} fails in this setting.     
}
\end{remark}

\noindent
{\bf Acknowledgement:} We thank Mihalis Sykiotis for pointing out some gaps in our arguments in an earlier version of this paper.

\end{document}